\documentclass[11pt, a4paper, leqno]{article}
\usepackage[margin=2.4cm]{geometry}
\usepackage[english]{babel}

\usepackage[runin]{abstract}
\usepackage{titling}

\usepackage{amsmath}
\usepackage{amsthm}
\usepackage{amsfonts}
\usepackage{amssymb}
\usepackage{eufrak}
\usepackage{stmaryrd}
\usepackage{bbm}
\usepackage{mathtools}
\usepackage{clrscode}
\usepackage{subscript}
\usepackage{enumitem}
\usepackage{mdwlist}
\usepackage{diagrams}

\abslabeldelim{.}

\newcommand{\keywordsname}{Key words}
\makeatletter
\newcommand{\keywords}[1]{%
\begin{@bstr@ctlist}
\hspace*{\abstitleskip}{\abstractnamefont\keywordsname\@bslabeldelim}\abstracttextfont\
#1%
\par\end{@bstr@ctlist}
}
\makeatother

\newcommand{\subjclassname}{Mathematics subject classification}
\makeatletter
\newcommand{\subjclass}[2][2010]{%
\begin{@bstr@ctlist}
\hspace*{\abstitleskip}{\abstractnamefont\subjclassname\ (#1)\@bslabeldelim}\abstracttextfont\
#2%
\par\end{@bstr@ctlist}
}
\makeatother

\makeatletter
\def\and{
	\end{tabular}%
	and%
	\begin{tabular}[t]{c}}%
\makeatother

\makeatletter
\def\thanks#1{
\protected@xdef\@thanks{\@thanks
\protect\footnotetext[\the\c@footnote]{#1}}%
}
\makeatother

\makeatletter
\let\addresses\@empty      
\newcommand{\address}[2][]{\g@addto@macro\addresses{\address{#1}{#2}}}
\newcommand{\curraddr}[2][]{\g@addto@macro\addresses{\curraddr{#1}{#2}}}
\newcommand{\email}[2][]{\g@addto@macro\addresses{\email{#1}{#2}}}
\newcommand{\urladdr}[2][]{\g@addto@macro\addresses{\urladdr{#1}{#2}}}
\def\enddoc@text{
  \ifx\@empty\addresses \else\@setaddresses\fi}
\AtEndDocument{\enddoc@text}
\def\emailaddrname{E-mail address}
\def\@setaddresses{\par
  \nobreak \begingroup
  \interlinepenalty\@M
  \def\address##1##2{\begingroup%
    \par\addvspace\bigskipamount
    \@ifnotempty{##1}{(\ignorespaces##1\unskip) }%
    {\noindent\ignorespaces##2}\par\endgroup}%
  \def\email##1##2{\begingroup
    \@ifnotempty{##2}{\nobreak\noindent\emailaddrname
      \@ifnotempty{##1}{, \ignorespaces##1\unskip}\/:\space
      \ttfamily##2\par}\endgroup}%
  \addresses
  \endgroup
}

\makeatother

\makeatletter
\def\cstar#1{\expandafter\@cstar\csname c@#1\endcsname}
\def\@cstar#1{\ifcase#1\or $\ast$\or $\ast\ast$\or $\ast\ast\ast$\fi}
\AddEnumerateCounter{\cstar}{\@cstar}{$\ast\ast\ast$}
\makeatother

\newlist{conditions}{enumerate}{1}
\newlist{Aconditions}{enumerate}{1}
\newlist{aconditions}{enumerate}{1}
\setlist[Aconditions]{label=\normalfont(A\textsubscript{\arabic*}),ref=\text{\normalfont(A\textsubscript{\arabic*})}}
\setlist[aconditions]{label=\normalfont(a\textsubscript{\arabic*}),ref=\text{\normalfont(a\textsubscript{\arabic*})}}
\setlist[conditions]{label=\normalfont(\alph*),ref=\normalfont\alph*}

\newcommand{\SC}{\mathcal{S}}
\newcommand{\Z}{\mathbb{Z}}
\newcommand{\R}{\mathbb{R}}
\newcommand{\C}{\mathcal{C}}
\newcommand{\RC}{\mathcal{R}}
\newcommand{\F}{\mathbb{F}}
\newcommand{\HB}{\mathbb{H}}
\newcommand{\CB}{\mathbb{C}}
\newcommand{\SB}{\mathbb{S}}
\newcommand{\K}{\mathbb{K}}
\newcommand{\G}{\mathbb{G}}
\newcommand{\cupproduct}{\mathbin{\smile}}
\newcommand{\Hsph}{H_{\mathrm{sph}}}
\newcommand{\Mod}{\func{mod}}
\newcommand{\NF}{\mathfrak{N}}

\newtheorem{conjecture}{Conjecture}

\newtheorem*{apprcon}{Approximation Conjecture}
\newtheorem{conjectureBp}{Conjecture}

\newtheorem{theorem}{Theorem}[section]
\newtheorem{proposition}[theorem]{Proposition}
\newtheorem{lemma}[theorem]{Lemma}
\theoremstyle{definition}
\newtheorem{example}[theorem]{Example}
\newtheorem{remark}[theorem]{Remark}

\DeclarePairedDelimiter\abs{\lvert}{\rvert}%
\DeclarePairedDelimiter\norm{\lVert}{\rVert}%

\makeatletter
\let\oldabs\abs
\def\abs{\@ifstar{\oldabs}{\oldabs*}}
\let\oldnorm\norm
\def\norm{\@ifstar{\oldnorm}{\oldnorm*}}
\makeatother

\mathchardef\mhyphen="2D

\title{\bf Some conjectures on continuous rational maps into spheres}
\date{}
\author{Wojciech Kucharz\thanks{The first author was partially supported
by the National Science Centre (Poland) under grant number
2014/15/B/ST1/00046. He also acknowledges with gratitude support and
hospitality of the Max--Planck--Institut f\"ur Mathematik in Bonn.}%
\and Krzysztof Kurdyka\thanks{The second author was
partially supported by ANR project STAAVF (France).}}

\address{Wojciech Kucharz\\Institute of Mathematics\\Faculty of Mathematics and Computer
Science\\Jagiellonian University\\\L{}ojasiewicza 6\\30-348
Krak\'ow\\Poland}
\email{Wojciech.Kucharz@im.uj.edu.pl}

\address{Krzysztof Kurdyka\\ Laboratoire de Math\'ematiques\\ UMR 5175
de CNRS\\ Universit\'e de Savoie\\ Campus Scientifique\\ 73 376 Le
Bourget-du-Lac Cedex\\ France}
\email{kurdyka@univ-savoie.fr}

\usepackage[pdftex, pdfauthor={\theauthor}, pdftitle={\thetitle}]{hyperref}

\begin{document}
\maketitle
\thispagestyle{empty}

\begin{abstract}
Recently continuous rational maps between real algebraic varieties have
attracted the attention of several researchers. In this paper we continue
the investigation of approximation properties of continuous rational
maps with values in spheres. We propose a conjecture concerning such
maps and show that it follows from certain classical conjectures
involving transformation of compact smooth submanifolds of nonsingular
real algebraic varieties onto subvarieties. Furthermore, we prove our
conjecture in a special case and obtain several related results.
\end{abstract}



\section{Introduction and main results}\label{sec-1}

Recently several authors devoted their papers to the investigation of
continuous rational maps between real algebraic varieties, cf.
\cite{bib3, bib7, bibstar, bib12, bib14, bib15, bib16, bib17, bib18, bib19}.
Continuing this line of research, we propose Conjecture~\ref{con-a},
whose proof would completely clarify many problems concerning
homotopical and approximation properties of continuous rational maps
with values in unit spheres. We prove this conjecture in a special case
and also obtain some related results. Furthermore, we show that
Conjecture~\ref{con-a} is a consequence of another conjecture, which has
nothing to do with continuous rational maps and originates from the
celebrated paper of Nash \cite{bib22} and the subsequent developments
due to Tognoli \cite{bib24}, Akbulut and King \cite{bib1}, and other
mathematicians. All results announced in this section are proved in
Section~\ref{sec-2}.

Throughout the present paper we use the term \emph{real algebraic
variety} to mean a locally ringed space isomorphic to an algebraic
subset of $\R^n$, for some $n$, endowed with the Zariski topology and
the sheaf of real-valued regular functions (such an object is called an
affine real algebraic variety in \cite{bib4}). The class of real
algebraic varieties is identical with the class of quasiprojective real
varieties, cf. \cite[Proposition~3.2.10, Theorem~3.4.4]{bib4}.
Nonsingular varieties are assumed to be of pure dimension. Morphisms of
real algebraic varieties are called \emph{regular maps}. Each real
algebraic variety carries also the Euclidean topology, which is induced
by the usual metric on $\R$. Unless explicitly stated otherwise, all
topological notions relating to real algebraic varieties refer to the
Euclidean topology.

Let $X$ and $Y$ be real algebraic varieties. A map $f \colon X \to Y$ is
said to be \emph{continuous rational} if it is continuous on $X$ and
there exists a Zariski open and dense subvariety $U$ of $X$ such that
the restriction $f|_U \colon U \to Y$ is a regular map. Let $X(f)$
denote the union of all such $U$. The complement $P(f) = X \setminus
X(f)$ of $X(f)$ is the smallest Zariski closed subvariety of~$X$ for
which the restriction $f|_{X \setminus P(f)} \colon X \setminus P(f) \to
Y$ is a regular map. If $f(P(f)) \neq Y$, we say that $f$ is a
\emph{nice} map. There exist continuous rational maps that are not nice,
cf. \cite[Example~2.2~(ii)]{bib14}. Continuous rational maps have only
recently become the object of serious investigation, cf. \cite{bib3,
bib7, bibstar, bib12, bib14, bib15, bib16, bib17, bib18, bib19}. They form a
natural intermediate class between regular and continuous maps. Having
many desirable features of regular maps they are more flexible.

The space $\C(X,Y)$ of all continuous maps from $X$ into $Y$ will always
be endowed with the compact-open topology. There are the following
inclusions
\begin{equation*}
\C(X,Y) \supseteq \RC^0(X,Y) \supseteq \RC_0(X,Y) \supseteq \RC(X,Y),
\end{equation*}
where $\RC^0(X,Y)$ is the set of all continuous rational maps,
$\RC_0(X,Y)$ consists of the nice maps in $\RC^0(X,Y)$, and $\RC(X,Y)$ is
the set of regular maps. By definition, a continuous map from $X$ into
$Y$ can be approximated by continuous rational maps if it belongs to the
closure of $\RC^0(X,Y)$ in $\C(X,Y)$. Approximation by nice continuous
rational maps or regular maps is defined in the analogous way.

Henceforth we assume that the variety $X$ is compact and nonsingular,
and concentrate our attention on maps with values in the unit
$p$-sphere
\begin{equation*}
\SB^p = \{ (u_0, \ldots, u_p) \in \R^{p+1} \mid u_0^2 + \cdots + u_p^2 = 1
\}
\end{equation*}
for $p \geq 1$. Regular maps from $X$ into $\SB^p$ have been extensively
studied, cf. \cite{bib4} and the literature cited there. Here we only
recall that the closure of $\RC(X, \SB^p)$ in $\C(X, \SB^p)$ can be a
much smaller set than the closure of $\RC_0(X,\SB^p)$, cf.
\cite[Example~1.8]{bib16}. If $\dim X \leq p$, then the set $\RC_0(X,
\SB^p)$ is dense in $\C(X, \SB^p)$. This assertion holds for $\dim X <p$
since $\R^p$ is biregularly isomorphic to~$\SB^p$ with one point
removed, whereas for $\dim X = p$ it is proved in \cite{bib16}.
Furthermore, $\RC_0(\SB^n, \SB^p)$ is dense in $\C(\SB^n, \SB^p)$ for all
positive integers $n$ and $p$, cf. \cite{bib16}. However, if $\dim X > p$,
then it can happen that a continuous map from $X$ into $\SB^p$ is not
homotopic to any continuous rational map, and hence $\RC^0(X, \SB^p)$ is
not dense in $\C(X, \SB^p)$, cf. \cite[Theorem~2.8]{bib16}. There are
reasons to believe that homotopical and approximation properties of nice
continuous rational maps from~$X$ into $\SB^p$, investigated in
\cite{bib14} and \cite{bib16}, are actually equivalent and fully
determined by certain (co)homological conditions. We give a precise
formulation of the last statement in Conjecture~\ref{con-a}.

Some preparation is required. Let $M$ be a compact smooth (of class
$\C^{\infty}$) codimension $p$ submanifold of $X$. If the normal bundle
to $M$ in $X$ is oriented, we denote by $\tau_M^X$ the Thom class of
$M$ in the cohomology group $H^p(X, X\setminus M; \Z)$, cf.
\cite[p.~118]{bib21}. The image of $\tau_M^X$ by the restriction
homomorphism $H^p(X, X \setminus M; \Z) \to H^p(X; \Z)$, induced by the
inclusion map $X \hookrightarrow (X, X \setminus M)$, will be denoted by
$\llbracket M \rrbracket^X$ and called the cohomology class represented
by $M$. If $X$ is oriented as a smooth manifold, then $\llbracket M
\rrbracket^X$ is up to sign Poincar\'e dual to the homology class in
$H_*(X; \Z)$ represented by $M$, cf. \cite[p.~136]{bib21}. Similarly,
without any orientability assumption, we define the cohomology class
$[M]^X$ in $H^p(X; \Z/2)$ represented by $M$. The cohomology class
$[M]^X$ is Poincar\'e dual to the homology class in $H_*(X; \Z/2)$
represented by $M$. Furthermore, if the normal bundle to $M$ in $X$ is
oriented, then
\begin{equation*}
\rho(\llbracket M \rrbracket^X) = [M]^X,
\end{equation*}
where
\begin{equation*}
\rho \colon H^*(X; \Z) \to H^*(X; \Z/2)
\end{equation*}
is the reduction $\Mod 2$ homomorphism.

We say that a cohomology class $u \in H^p(X; \Z)$ (resp. $v \in H^p(X;
\Z/2)$) is \emph{adapted} if there exists a nonsingular codimension $p$
Zariski locally closed subvariety $Z$ of $X$ such that $Z$ is a~compact
smooth submanifold with trivial normal bundle and
\begin{equation*}
u = \llbracket Z \rrbracket^X \quad\textrm{(resp. $v=[Z]^X$)},
\end{equation*}
where the first equality holds when the normal bundle to $Z$ is suitably
oriented. Here $Z$ need not be Zariski closed in $X$, but the
nonsingular locus of its Zariski closure coincides with $Z$. Denote by
$A^p(X; \Z)$ (resp. $A^p(X; \Z/2)$) the subgroup of $H^p(X; \Z)$ (resp.
$H^p(X; \Z/2)$) generated by all adapted cohomology classes. By
construction,
\begin{equation*}
\rho(A^p(X; \Z)) = A^p(X; \Z/2).
\end{equation*}

The groups $A^p(-;\Z)$ and $A^p(-;\Z/2)$ can be explicitly computed for
some real algebraic varieties.

\begin{example}\label{ex-1-1}
Let $X = X_1 \times \cdots \times X_r$, where $X_i$ is a nonsingular
real algebraic variety diffeomorphic to the $n_i$-sphere for $1 \leq i
\leq r$. Then, by the K\"unneth formula,
\begin{equation*}
A^p(X; \Z) = H^p(X;\Z) \quad \textrm{and}\quad A^p(X; \Z/2) = H^p(X;
\Z/2)
\end{equation*}
for every $p \geq 0$.
\end{example}

Let $s_p$ (resp. $\bar{s}_p$) be a generator of the cohomology group
$H^p(\SB^p; \Z) \cong \Z$ (resp.\linebreak $H^p(\SB^p; \Z/2) \cong \Z/2)$, $p \geq
1$. Recall that a cohomology class $u \in H^p(X; \Z)$ (resp. ${v \in
H^p(X; \Z/2)}$) is said to be \emph{spherical} if $u = f^*(s_p)$ (resp.
$v = f^*(\bar{s}_p)$) for some continuous map $f \colon X \to \SB^p$.
Without loss of generality, the map $f$ can be assumed to be smooth. In
that case, if $y \in \SB^p$ is a~regular value of $f$, then the inverse
image $f^{-1}(y)$ is a compact smooth codimension $p$ submanifold of
$X$, embedded with trivial normal bundle, for which
\begin{equation*}
f^*(s_p) = \llbracket f^{-1}(y) \rrbracket^X \quad\textrm{and}\quad
f^*(\bar{s}_p) = [f^{-1}(y)]^X,
\end{equation*}
where the first equality holds provided that the normal bundle to
$f^{-1}(y)$ in $X$ is suitably oriented (this is a well known fact whose
proof is recalled in \cite[p.~258]{bib17}). Conversely, if $M$ is
a~compact smooth codimension $p$ submanifold of $X$ with normal bundle
trivial and oriented, then the cohomology classes $\llbracket M
\rrbracket^X \in H^p(X; \Z)$ and $[M]^X \in H^p(X; \Z/2)$ are spherical
(this is a consequence of a classical result in framed cobordism, cf.
\cite[p.~44]{bib20}). We denote by $\Hsph^p(X; \Z)$ (resp. $\Hsph^p (X;
\Z/2)$) the subgroup of $H^p(X;\Z)$ (resp. $H^p(X; \Z/2)$) generated by
all spherical cohomology classes. As explained above,
\begin{equation*}
A^p(X; \Z) \subseteq \Hsph^p(X; \Z) \quad\textrm{and}\quad A^p(X; \Z/2)
\subseteq \Hsph^p(X; \Z/2).
\end{equation*}

\begin{conjecture}\label{con-a}
Let $X$ be a compact nonsingular real algebraic variety and let $p$ be a
positive integer. For a continuous map $f \colon X \to \SB^p$, the
following conditions are equivalent:
\begin{Aconditions}
\item\label{A1} $f$ can be approximated by nice continuous rational
maps.

\item\label{A2} $f$ is homotopic to a nice continuous rational map.

\item\label{A3} The cohomology class $f^*(s_p) \in H^p(X; \Z)$ is
adapted.

\item\label{A4} The cohomology class $f^*(\bar{s}_p) \in H^p(X; \Z/2)$
is adapted.

\item\label{A5} $f^*(s_p) \in A^p(X; \Z)$.

\item\label{A6} $f^*(\bar{s}_p) \in A^p(X; \Z/2)$.
\end{Aconditions}
\end{conjecture}

The known implications between conditions \ref{A1} through \ref{A6} are
indicated as follows
\begin{diagram}[small]
\ref{A1} & \rImplies & \ref{A2} & \rImplies & \ref{A3} & \rImplies & \ref{A4}\\
& & & & \dImplies & &\dImplies \\
& & & & \ref{A5} & \rImplies & \ref{A6}
\end{diagram}
The implication \ref{A2} $\Rightarrow$ \ref{A3} is proved in
\cite[Proposition~1.1]{bib17}, while all the others are obvious. The maps
satisfying \ref{A1} (resp. \ref{A2}) are characterized in \cite{bib16}
(resp. \cite{bib14}). The results of \cite{bib14, bib16} are crucial in
the proof of
\begin{equation*}
\textrm{\ref{A2}} \Rightarrow \textrm{\ref{A1} for } \dim X + 3 \leq 2p,
\end{equation*}
given in \cite{bib18}.

\begin{remark}\label{rem-1-2}
In order to prove Conjecture~\ref{con-a} it would suffice to show that
\ref{A6} implies \ref{A1}.
\end{remark}

We record the following observation.

\begin{remark}\label{rem-1-3}
Let $X$ be a compact nonsingular real algebraic variety and let $p$ be a
positive integer. If Conjecture~\ref{con-a} holds, then the following
conditions are equivalent:
\begin{conditions}
\item\label{1-3-a} The set $\RC_0(X, \SB^p)$ of nice continuous rational
maps
is dense in $\C(X, \SB^p)$.

\item\label{1-3-b} Each continuous map from $X$ into $\SB^p$ is
homotopic to a nice continuous rational map.

\item\label{1-3-c} $A^p(X; \Z) = \Hsph^p (X; \Z)$.

\item\label{1-3-d} $A^p(X; \Z/2) = \Hsph^p (X; \Z/2)$.
\end{conditions}
\end{remark}

As already mentioned at the beginning of this section,
Conjecture~\ref{con-a} is related to seemingly quite different problems
investigated in \cite{bib1, bib22, bib24}. We now explain this in
detail.

Let $V$ be a real algebraic variety. A bordism class in the $n$th
unoriented bordism group $\NF_n(V)$ of $V$ is said to be
\emph{algebraic} if it can be represented by a regular map from an
$n$-dimensional compact nonsingular real algebraic variety into $V$, cf.
\cite{bib1, bib2}.

\begin{apprcon}
For any nonsingular real algebraic variety $V$, the following condition
is satisfied: If $M$ is a compact smooth submanifold of $V$ and the
unoriented bordism class of the inclusion map $M \hookrightarrow X$ is
algebraic, then $M$ is $\varepsilon$-isotopic to a nonsingular Zariski
locally closed subvariety of $V$.
\end{apprcon}

Here ``$\varepsilon$-isotopic'' means isotopic via a smooth isotopy that
can be chosen arbitrarily close, in the $\C^{\infty}$ topology, to the
inclusion map $M \hookrightarrow V$.  A slightly weaker assertion than the
one in the Approximation Conjecture is known to be true: If the
unoriented bordism class of the inclusion map $M \hookrightarrow V$ is
algebraic, then the smooth submanifold $M \times \{ 0 \}$ of $V \times
\R$ is $\varepsilon$-isotopic to a nonsingular Zariski locally closed
subvariety of $V \times \R$, cf. \cite[Theorem~F]{bib1}.

The following is a special case of the Approximation Conjecture.

\stepcounter{conjecture}
\begin{conjectureBp}\label{con-bp}
For any compact nonsingular real algebraic variety $V$, the following
condition is satisfied: If $M$ is a compact smooth codimension $p$
submanifold of $V$, embedded with trivial normal bundle, and the
unoriented bordism class of the inclusion map is algebraic, then $M$ is
$\varepsilon$-isotopic to a nonsingular Zariski locally closed subvariety
of $V$.
\end{conjectureBp}

Presumably, Conjecture~\ref{con-bp} should be easier to prove than the
Approximation Conjecture. In the context of this paper,
Conjecture~\ref{con-bp} is of particular interest.

\begin{proposition}\label{prop-1-4}
If Conjecture~\ref{con-bp} holds, then so does Conjecture~\ref{con-a}.
\end{proposition}

Using a method independent of Conjecture~\ref{con-bp}, we prove the
following special case of Conjecture~\ref{con-a}.

\begin{theorem}\label{th-1-5}
Let $X$ be a compact nonsingular real algebraic variety of dimension
$n$. Let $d$ and $p$ be positive integers satisfying $n+1 \leq p$ and
$n+d+2 \leq 2p$. Then for a continuous map $f \colon X \times \SB^d \to
\SB^p$, the following conditions are equivalent:
\begin{aconditions}
\item\label{a1} $f$ can be approximated by nice continuous rational
maps.

\item\label{a2} $f$ is homotopic to a nice continuous rational map.

\item\label{a3} The cohomology class $f^*(s_p) \in H^p(X \times \SB^d;
\Z)$ is adapted.

\item\label{a4} The cohomology class $f^*(\bar{s}_p) \in H^p(X \times
\SB^d; \Z/2)$ is adapted.

\item\label{a5} $f^*(s_p) \in A^p(X \times \SB^d; \Z)$.

\item\label{a6} $f^*(\bar{s}_p) \in A^p(X \times \SB^d; \Z/2)$.
\end{aconditions}
\end{theorem}\pagebreak

For any positive integers $n$, $d$ and $p$, we have
\begin{equation*}
A^p (\SB^n \times \SB^d; \Z/2) = H^p(\SB^n \times \SB^d; \Z/2)
\end{equation*}
and hence, if Conjecture~\ref{con-a} holds, then the set $\RC_0(\SB^n
\times \SB^d, \SB^p)$ of nice continuous rational maps is dense in
$\C(\SB^n \times \SB^d, \SB^p)$. Making no use of
Conjecture~\ref{con-a}, we obtain the following weaker result.

\begin{example}\label{ex-1-6}
If $n$, $d$ and $p$ are positive integers satisfying $n+d+2 \leq 2p$,
then $\RC_0(\SB^n \times \SB^d, \SB^p)$ is dense in $\C(\SB^n \times
\SB^d, \SB^p)$. Indeed, we may assume that $n \leq d$. Then $n+1 \leq p$
and hence the assertion follows from Theorem~\ref{th-1-5}.
\end{example}

Some special techniques are available when one studies regular or
continuous rational maps with values in $\SB^p$ for $p$ equal to $1$,
$2$ or $4$, cf. \cite{bib4, bib5, bib14, bib15, bib16, bib17, bib19}.
According to \cite[Theorem~1.2]{bib17}, each continuous map $f \colon X
\to \SB^2$ with $f^*(s_2) \in A^2(X; \Z)$ can be approximated by
continuous rational maps. For maps with values in $\SB^4$ we have the
following result.

\begin{theorem}\label{th-1-7}
Let $X$ be a compact nonsingular real algebraic variety. Assume that for
each integer $k \geq 3$, the only torsion in the cohomology group
$H^{2k}(X;\Z)$ is relatively prime to $(k-1)!$. A continuous map $f
\colon X \to \SB^4$ can be approximated by continuous rational maps,
provided that the cohomology class $f^*(s_4) \in H^4(X;\Z)$ is adapted.
\end{theorem}

The assumption in Theorem~\ref{th-1-7} that the cohomology class
$f^*(s_4) \in H^4(X;\Z)$ is adapted is not very convenient and it would
be desirable to replace it by the condition $f^*(s_4) \in A^4(X; \Z)$.
This can be done at least for $\dim X \leq 7$.

\begin{theorem}\label{th-1-8}
Let $X$ be a compact nonsingular real algebraic variety of dimension at
most $7$. Assume that the cohomology group $H^6(X;\Z)$ has no
$2$-torsion. A continuous map $f \colon X \to \SB^4$ can be approximated
by continuous rational maps, provided that $f^*(s_4) \in A^4(X;\Z)$.
\end{theorem}

One might wonder whether the condition ${f^*(s_4) \in A^4(X; \Z)}$ can be
replaced by 
\begin{equation*}
f^*(\bar{s}_4) \in A^4(X; \Z/2).
\end{equation*}
We only have a partial
result.

\begin{theorem}\label{th-1-9}
Let $X$ be a compact nonsingular real algebraic variety of dimension at
most~$7$. Assume that the cohomology group $H^6(X;\Z)$ has no
$2$-torsion and $\Hsph^4(X;\Z) = H^4(X;\Z)$. A~continuous map $f \colon
X \to \SB^4$ can be approximated by continuous rational maps, provided that
\begin{equation*}
f^*(\bar{s}_4) \in A^4(X; \Z/2).
\end{equation*}
\end{theorem}

The last three theorems can be strengthened if the following holds.

\begin{conjecture}\label{con-c}
Any continuous rational map from a compact nonsingular real algebraic
variety~$X$ into the $p$-sphere can be approximated by nice continuous
rational maps.
\end{conjecture}

Conjecture~\ref{con-c} is known to hold if $p=1$ or $\dim X \leq p+1$,
cf. \cite{bib16}.

\section{Proofs}\label{sec-2}

For the clarity of the exposition we recall the following approximation
criterion.

\begin{theorem}[{\cite[Theorem~1.2]{bib16}}]\label{th-2-1}
Let $X$ be a compact nonsingular real algebraic variety and let $f
\colon X \to \SB^p$ be a smooth map. Assume that there exists a regular
value $y \in \SB^p$ of $f$ such that the smooth submanifold $f^{-1}(y)$
is $\varepsilon$-isotopic to a nonsingular Zariski locally closed
subvariety of~$X$. Then $f$ can be approximated by nice continuous
rational maps.
\end{theorem}

The proof of Proposition~\ref{prop-1-4} will be preceded by two lemmas.
As usual, given a smooth manifold $M$, we denote by $w_i(M)$ its $i$th
Stiefel--Whitney class. If the manifold $M$ is compact, let $[M]$ denote
its fundamental class in the homology group $H_*(M; \Z/2)$.

\begin{lemma}\label{lem-2-2}
Let $N$ and $P$ be compact smooth manifolds, and let $f \colon N \to P$
be a smooth immersion with trivial normal bundle. If $f_*([N]) = 0$ in
$H_*(P; \Z/2)$, then the unoriented bordism class of the map $f$ is
zero.
\end{lemma}

\begin{proof}
According to \cite[(17.3)]{bib6}, it suffices to prove that for any
nonnegative integer $l$ and any cohomology class $u$ in $H^l(P; \Z/2)$,
the equality
\begin{equation}\label{eq-dagger}\tag{$\dagger$}
\langle w_{i_1}(N) \cupproduct \cdots \cupproduct w_{i_r}(N) \cupproduct
f^*(u), [N] \rangle = 0
\end{equation}
holds for all nonnegative integers $i_1, \ldots, i_r$ satisfying $i_1 +
\cdots + i_r = \dim N - l$.

Since the normal bundle of the immersion $f$ is trivial, we have
\begin{equation*}
w_i(N) = f^*(w_i(P))
\end{equation*}
for every nonnegative integer $i$, cf. \cite[p.~31]{bib21}. Setting
\begin{equation*}
v = w_{i_1}(P) \cupproduct \cdots \cupproduct w_{i_r}(P) \cupproduct u,
\end{equation*}
we get
\begin{equation*}
w_{i_1}(N) \cupproduct \cdots \cupproduct w_{i_r}(N) \cupproduct f^*(u)
= f^*(v).
\end{equation*}
Since
\begin{equation*}
\langle f^*(v), [N] \rangle = \langle v, f_*([N]) \rangle,
\end{equation*}
equality \eqref{eq-dagger} holds if $f_*([N]) = 0$. The proof is complete.
\end{proof}

\begin{lemma}\label{lem-2-3}
Let $X$ be a compact nonsingular real algebraic variety. Let $M$ be a
compact smooth codimension $p$ submanifold of $X$, embedded with trivial
normal bundle. Assume that the cohomology class $[M]^X$ belongs to
$A^p(X; \Z/2)$. Then the unoriented bordism class of the inclusion map
$e \colon M \hookrightarrow X$ is algebraic.
\end{lemma}

\begin{proof}
By the definition of $A^p(X; \Z/2)$, we have
\begin{equation*}
[M]^X = [Z_1]^X + \cdots + [Z_k]^X,
\end{equation*}
where $Z_i$ is a nonsingular codimension $p$ Zariski locally closed
subvariety of $X$, which is a compact smooth submanifold with trivial
normal bundle, $1 \leq i \leq k$. Let $Z$ be the disjoint union of
the~$Z_i$, and let $g \colon Z \to X$ be the map whose restriction to
$Z_i$
corresponds to the inclusion map $Z_i \hookrightarrow X$. By
construction, $Z$ is a compact nonsingular real algebraic variety and $g
\colon Z \to X$ is a regular map. Since the cohomology class $[M]^X$
(resp. $[Z_1]^X + \cdots + [Z_k]^X$) is Poincar\'e dual to the
cohomology class $e_*([M])$ (resp. $g_*([Z])$), we get
\begin{equation*}
e_*([M]) = g_*([Z]).
\end{equation*}
Let $N$ be the disjoint union of $M$ and $Z$, and let $f \colon N \to X$
be the map whose restriction to $M$ (resp. $Z$) corresponds to $e$ (resp.
$g$). Note that $f \colon N \to X$ is a smooth immersion with trivial
normal bundle. Furthermore, $f_*([N]) = 0$. In view of
Lemma~\ref{lem-2-2}, the maps $e \colon M \hookrightarrow X$ and $g
\colon Z \to X$ represent the same unoriented bordism class, which
completes the proof.
\end{proof}

\begin{proof}[Proof of Proposition~\ref{prop-1-4}]
Suppose that Conjecture~\ref{con-bp} holds. It suffices to show that
\ref{A6} implies \ref{A1}. Assume that \ref{A6} is satisfied, that is,
\begin{equation*}
f^*(\bar{s}_p) \in A^p(X; \Z/2).
\end{equation*}
We can assume without loss of generality that the map $f \colon X \to
\SB^p$ is smooth. By Sard's theorem there exists a regular value $y \in \SB^p$ of
$f$. The inverse image $f^{-1}(y)$ is a compact smooth submanifold of
$X$, embedded with trivial normal bundle. Since
\begin{equation*}
f^*(\bar{s}_p) = [f^{-1}(y)]^X,
\end{equation*}
it follows from Lemma~\ref{lem-2-3} that the unoriented bordism class of
the inclusion map $f^{-1}(y) \hookrightarrow X$ is algebraic. According
to Conjecture~\ref{con-bp}, the smooth submanifold $f^{-1}(y)$ is
$\varepsilon$-isotopic to a nonsingular Zariski locally closed
subvariety of $X$. Hence, in view of Theorem~\ref{th-2-1}, condition
\ref{A1} is satisfied.
\end{proof}

The proof of Theorem~\ref{th-1-5} requires some preparation.

\begin{lemma}\label{lem-2-4}
Let $Y$ be a nonsingular real algebraic variety and let $M$ be a compact
smooth submanifold of $Y$, embedded with trivial normal bundle. If $M$
is isotopic to a nonsingular Zariski closed subvariety of $Y$, then it
is $\varepsilon$-isotopic to a nonsingular Zariski closed subvariety of
$Y$.
\end{lemma}

\begin{proof}
This is proved in \cite[Theorem~2.1]{bib13} if the ambient variety $Y$ is
compact. In the general case we can argue as follows. Suppose that $M$ is
isotopic to a nonsingular Zariski closed subvariety $V$ of $Y$.
According to Hironaka's theorem \cite{bib8} (cf. also \cite{bib11} for a
very readable exposition), we can assume that $Y$ is a Zariski open
subvariety of a compact nonsingular real algebraic variety~$X$. Let $A$
be the Zariski closure of $V$ in $X$. Then $S \coloneqq A \setminus V$ is
a Zariski closed subvariety of $X$, contained in $X \setminus Y$. By
Hironaka's theorem, there exists a regular map $\pi \colon X' \to X$
such that $X'$ is a compact nonsingular real algebraic variety, the
restriction $\pi_0 \colon X' \setminus \pi^{-1}(S) \to X \setminus S$ of
$\pi$ is a~biregular isomorphism, and the subvariety $\pi^{-1}(V)$ is
Zariski closed in $X'$. Identifying $Y$, $M$, $V$ with $\pi^{-1}(Y)$,
$\pi^{-1}(M)$, $\pi^{-1}(V)$, respectively, we can assume that $V$ is
Zariski closed in $X$. The proof is complete since $X$ is compact and
$M$ is isotopic to $V$ in $X$.
\end{proof}

\begin{lemma}\label{lem-2-5}
Let $X$ be a nonsingular real algebraic variety and let $d$ be a
positive integer. Let $M$ be a compact smooth submanifold of $X \times
\SB^d$, embedded with trivial normal bundle. Assume that
\begin{equation*}
2\dim M + 2 \leq \dim X + d,
\end{equation*}
the unoriented bordism class of the inclusion map $M \hookrightarrow X
\times \SB^d$ is algebraic, and $\sigma(M) \neq \SB^d$, where $\sigma
\colon X \times \SB^d \to \SB^d$ is the canonical projection. Then the
smooth submanifold $M$ is $\varepsilon$-isotopic to a nonsingular
Zariski locally closed subvariety of $X \times \SB^d$.
\end{lemma}

\begin{proof}
The assumption ${\sigma(M) \neq \SB^p}$ implies the existence of a point
${u \in \SB^p}$ for which
\begin{equation*}
M \subseteq X \times (\SB^d \setminus \{u\}).
\end{equation*}
Since $\SB^d \setminus \{u\}$ is biregularly isomorphic to $\R^d$, we
identify $\SB^d \setminus \{u\}$ with $\R^d$ and regard $M$ as a
submanifold of $X \times \R^d$. It remains to prove that $M$ is
$\varepsilon$-isotopic to a Zariski locally closed subvariety of $X
\times \R^d$.

Let $\varphi \colon M \to X \times \R^{d-1}$ and $\psi \colon M \to \R$
be maps such that
\begin{equation*}
f \coloneqq (\varphi, \psi) \colon M \to (X \times \R^{d-1}) \times \R =
X \times \R^d
\end{equation*}
is the inclusion map. Since
\begin{equation*}
2 \dim M + 1 \leq \dim X + d -1,
\end{equation*}
the map $\varphi$ is homotopic to a smooth embedding $\bar{\varphi}
\colon M \to X \times \R^{d-1}$, cf. \cite[p.~55]{bib9}. If $\bar{\psi}
\colon M \to \R$ is the constant map sending $M$ to $0 \in \R$, then the
map
\begin{equation*}
\bar{f} \coloneqq (\bar{\varphi}, \bar{\psi}) \colon M \to (X \times
\R^{d-1}) \times \R) = X \times \R^d
\end{equation*}
is a smooth embedding homotopic to $f$. Since
\begin{equation*}
2\dim M + 2 \leq \dim X + d,
\end{equation*}
the smooth embeddings $f$ and $\bar{f}$ are isotopic, cf.
\cite[Theorem~6]{bib25} or \cite[p.~183, Exercise~10]{bib9}. Thus $M$ is
isotopic in $(X \times \R^{d-1}) \times \R$ to the smooth submanifold
$\bar{f}(M) = \bar{M} \times \{0\}$, where $\bar{M} \coloneqq
\bar{\varphi}(M)$.

We assert that the smooth submanifold $\bar{f}(M)$ is
$\varepsilon$-isotopic to a nonsingular Zariski locally closed
subvariety of $X \times \R^d$. This can be proved as follows. Since the
unoriented bordism class of the inclusion map $e \colon M
\hookrightarrow X \times \SB^d$ is algebraic, so is the unoriented
bordism class of the map $e_X \colon M \to X$, where $e_X$ is
the composite of $e$ and the canonical projection $X \times \SB^d \to
X$. It follows that the unoriented bordism class of the map $\varphi
\colon M \to X \times \R^{d-1}$ is algebraic. The maps $\varphi \colon M
\to X \times \R^{d-1}$ and $\bar{\varphi} \colon M \to X \times
\R^{d-1}$ represent the same unoriented bordism class. Consequently, the
unoriented bordism class of the inclusion map $\bar{M} \hookrightarrow X
\times \R^{d-1}$ is algebraic. This implies the assertion in view of
\cite[Theorem~F]{bib1}.

By the assertion, $M$ is isotopic to a nonsingular Zariski locally
closed subvariety $Z$ of $X \times \R^d$. Let $F \colon M \times [0,1]
\to X \times \R^d$ be a smooth isotopy such that $F_0$ is the inclusion
map and $F_1(M) = Z$, where $F_t(x) = F(x,t)$ for all $x \in M$ and $t
\in [0,1]$. Let $A$ be the Zariski closure of $Z$ in $X \times \R^d$.
Then $S \coloneqq A \setminus Z$ is a Zariski closed subvariety of $X
\times \R^d$ of dimension at most $\dim Z -1 = \dim M -1$. In
particular, $S$ has a finite stratification into smooth submanifolds of
$X \times \R^d$ of dimension at most $\dim S$. Since
\begin{equation*}
\dim (M \times [0, 1]) + \dim S \leq 2\dim M < \dim X + d,
\end{equation*}
according to the transversality theorem, there exists a smooth map $G
\colon M \times [0,1] \to X \times \R^d$, arbitrarily close to $F$ in
the $\C^{\infty}$ topology, such that
\begin{equation*}
G(M \times [0,1]) \subseteq (X \times \R^d) \setminus S
\quad\textrm{and}\quad G_1 = F_1.
\end{equation*}
Note that $G_1(M) = Z$ is a Zariski closed subvariety of $(X \times
\R^d) \setminus S$. If $G$ is close enough to~$F$, then~$G$ is an
isotopy. Consequently, the smooth submanifold $N \coloneqq G_0(M)$ is
isotopic to $X = G_1(M)$ in $(X \times \R^d) \setminus S$. The normal
bundle to $N$ in $(X \times \R^d) \setminus S$ is trivial and hence, in
view of Lemma~\ref{lem-2-4}, $N$ is $\varepsilon$-isotopic to a
nonsingular Zariski closed subvariety of $(X \times \R^d) \setminus S$.
Since the smooth embedding $G_0 \colon M \to X \times \R^d$ is close to
the inclusion map $F_0 \colon M \hookrightarrow X \times \R^d$, it
follows that $M$ is $\varepsilon$-isotopic to a nonsingular Zariski
locally closed subvariety of $X \times \R^d$, as required.
\end{proof}

\begin{proof}[Proof of Theorem~\ref{th-1-5}]
We can assume without loss of generality that the map $f$ is smooth. It
suffices to prove that \ref{a6} implies \ref{a1}. Suppose that \ref{a6}
holds.

Let $u$ be a point in $\SB^d$. Since $n+1 \leq p$, we have
\begin{equation*}
f(X \times \{u\}) \neq \SB^p.
\end{equation*}
By Sard's theorem, there exists a regular value $y \in \SB^p \setminus
f(X \times \{u\})$ of $f$. Then $M \coloneqq f^{-1}(y)$ is a compact
smooth submanifold of $X \times \SB^d$, embedded with trivial normal
bundle. Since $[M]^X = f^*(\bar{s}_p)$, it follows from
Lemma~\ref{lem-2-3} that the unoriented bordism class of the inclusion
map ${M \hookrightarrow X \times \SB^d}$ is algebraic. Note that
$\sigma(M) \subseteq \SB^d \setminus \{u\}$, where $\sigma \colon X
\times \SB^d \to \SB^d$ is the canonical projection. Furthermore, $\dim M
= n + d - p$, and hence
\begin{equation*}
2\dim M + 2 \leq n+d.
\end{equation*}
Consequently, according to Lemma~\ref{lem-2-5}, the submanifold $M$ is
$\varepsilon$-isotopic to a nonsingular Zariski locally closed
subvariety of $X \times \SB^d$. Thus, \ref{a1} holds in view of
Theorem~\ref{th-2-1}.
\end{proof}

In the remainder of this paper we will need several results on
stratified-algebraic vector bundles, all of which are proved in
\cite{bib19}.

Let $X$ be a real algebraic variety. By a \emph{stratification} of $X$
we mean a finite collection $\SC$ of pairwise disjoint Zariski locally
closed subvarieties whose union is $X$. Each subvariety in $\SC$ is
called a stratum of $\SC$. A map $f \colon X \to Y$, where $Y$ is a real
algebraic variety, is said to be \emph{stratified-regular} if it is
continuous and for some stratification $\SC$ of $X$, the restriction
$f|_S \colon S \to Y$ of $f$ to each stratum $S$ in $\SC$ is a regular
map, cf. \cite{bib19}. The notion of stratified-regular map is closely
related to those of hereditarily rational function \cite{bib12} and
fonction r\'egulue \cite{bib7}. One readily sees that each
stratified-regular map is continuous rational.

Let $\F$ stand for $\R$, $\CB$ or $\HB$ (the quaternions). All
$\F$-vector spaces will be left $\F$-vector spaces. When convenient,
$\F$ will be identified with $\R^{d(\F)}$, where $d(\F) = \dim_{\R} \F$.

For any nonnegative integer $n$, let $\varepsilon_X^n(\F)$ denote the
standard trivial $\F$-vector bundle on $X$ with total space $X \times
\F^n$, where $X \times \F^n$ is regarded as a real algebraic variety.

An algebraic $\F$-vector bundle on $X$ is an algebraic $\F$-vector
subbundle of $\varepsilon_X^n(\F)$ for some $n$ (cf. \cite[Chapters~12
and~13]{bib4} for various characterizations of algebraic $\F$-vector
bundles).

We now recall the fundamental notion introduced in \cite{bib19}. A
\emph{stratified-algebraic} $\F$-vector bundle on $X$ is a topological
$\F$-vector subbundle $\xi$ of $\varepsilon_X^n(\F)$, for some $n$, such
that for some stratification $\SC$ of $X$, the restriction $\xi|_S$ of
$\xi$ to each stratum $S$ in $\SC$ is an algebraic $\F$-vector subbundle
of $\varepsilon_S^n(\F)$.

A topological $\F$-vector bundle on $X$ is said to \emph{admit a
stratified-algebraic structure} if it is isomorphic to a
stratified-algebraic $\F$-vector bundle on $X$.

Let $\K$ be a subfield of $\F$, where $\K$ (as $\F$) stands for $\R$,
$\CB$ or $\HB$. Any $\F$-vector bundle $\xi$ on $X$ can be regarded as a
$\K$-vector bundle, which is indicated by $\xi_{\K}$. In particular,
$\xi_{\K} = \xi$ if $\K=\F$. If the $\F$-vector bundle $\xi$ admits a
stratified-algebraic structure, then so does the $\K$-vector
bundle~$\xi_{\K}$. The following result will play a crucial role.

\begin{theorem}[{\cite[Theorem~1.7]{bib19}}]\label{th-2-6}
Let $X$ be a compact real algebraic variety. A topological $\F$-vector
bundle $\xi$ on $X$ admits a stratified-algebraic structure if and only
if the $\K$-vector bundle~$\xi_{\K}$ admits a stratified-algebraic
structure.
\end{theorem}

The proof for $\K=\R$, rather involved, is given in \cite{bib19}. The
general case follows since $\xi_{\R} = (\xi_{\K})_{\R}$. In the present
paper we make use of this result with $\F=\HB$ and $\K=\CB$.

Another useful fact is the following.

\begin{theorem}[{\cite[Corollary~3.14]{bib19}}]\label{th-2-7}
Let $X$ be a compact real algebraic variety. A topological $\F$-vector
bundle on $X$ admits a stratified-algebraic structure if and only if it
is stably equivalent to a stratified-algebraic $\F$-vector bundle on
$X$.
\end{theorem}

There is a close connection between stratified-algebraic vector bundles
and approximation by stratified-regular maps with values in
Grassmannians. Let $\G_k(\F^n)$ denote the Grassmannian of
$k$-dimensional $\F$-vector subspaces of $\F^n$. We regard $\G_k(\F^n)$
as a real algebraic variety, cf.~\cite{bib4}. The tautological
$\F$-vector bundle $\gamma_k(\F^n)$ on $\G_k(\F^n)$ is algebraic.

\begin{theorem}[{\cite[Theorem~4.10]{bib19}}]\label{th-2-8}
Let $X$ be a compact real algebraic variety. For a continuous map $f
\colon X \to \G_k(\F^n)$, the following conditions are equivalent:
\begin{conditions}
\item\label{th-2-8-a} $f$ can be approximated by stratified-regular
maps.

\item\label{th-2-8-b} $f$ is homotopic to a stratified-regular map.

\item\label{th-2-8-c} The pullback $\F$-vector bundle
$f^*\gamma_k(\F^n)$ on $X$ admits a stratified-algebraic structure.
\end{conditions}
\end{theorem}

As usual, the $k$th Chern class of a $\CB$-vector bundle $\xi$ will be
denoted by $c_k(\xi)$. Note that if $\eta$ is an $\HB$-vector bundle,
then
\begin{equation*}
c_{2l+1}(\eta_{\CB}) = 0
\end{equation*}
for every $l \geq 0$.\pagebreak

We now proceed to the investigation of maps with values in $\SB^4$.
Recall that the Grassmannian $\G_1(\HB^2)$ is biregularly isomorphic to
$\SB^4$. Henceforth we identify $\G_1(\HB^2)$ with $\SB^4$ and set
$\gamma = \gamma_1(\HB^2)$. In particular, $\gamma$ is an algebraic
$\HB$-line bundle on $\SB^4$ and the Chern class $c_2(\gamma_{\CB})$ is
a generator of the cohomology group $H^4(\SB^4; \Z)$. We can assume
without loss of generality that
\begin{equation*}
s_4 = c_2(\gamma_{\CB}).
\end{equation*}

\begin{lemma}\label{lem-2-9}
Let $X$ be a compact nonsingular real algebraic variety and let $u \in
H^4(X; \Z)$ be an adapted cohomology class. Then there exists a
stratified-algebraic $\HB$-line bundle $\xi$ on $X$ with $c_2(\xi_{\CB})
= u$.
\end{lemma}

\begin{proof}
The cohomology class $u$ is of the form
\begin{equation*}
u = \llbracket Z \rrbracket^X,
\end{equation*}
where $Z$ is a nonsingular codimension $4$ Zariski locally closed
subvariety of $X$, which is a compact smooth submanifold with trivial
normal bundle that is suitably oriented. Choose a smooth framing $F$ of
the normal bundle to $Z$ in $X$ so that it determines the existing
orientation. According to a classical result in framed cobordism, we can
find a smooth map $f \colon X \to \SB^4$ and a regular value $y \in
\SB^4$ of $f$ such that the framed submanifolds $(Z, F)$ and
$(f^{-1}(y), F_f)$ are equal, where $F_f$ is a framing of $f^{-1}(y)$ induced by
$f$, cf. \cite[p.~44]{bib20}. Then
\begin{equation*}
f^*(s_4) = \llbracket Z \rrbracket^X = u.
\end{equation*}
Furthermore, according to \cite[Theorem~2.4]{bib14}, $f$ is homotopic to a
continuous rational map ${g \colon X \to \SB^4}$. In particular,
\begin{equation*}
g^*(s_4) = f^*(s_4) = u.
\end{equation*}
Since the variety $X$ is nonsingular, the map $g$ is stratified-regular
(this follows from \cite[Proposition~8]{bib12} as commented in
\cite[Remark~2.3]{bib19}). The pullback $\HB$-line bundle
\begin{equation*}
\xi \coloneqq g^*\gamma
\end{equation*}
on $X$ is stratified-regular, the map $g$ being stratified-regular and
the $\HB$-line bundle $\gamma$ on $\SB^4$ being algebraic. Furthermore,
\begin{equation*}
c_2(\xi_{\CB}) = c_2(g^*\gamma_{\CB}) = g^*(c_2(\gamma_{\CB})) = g^*(s_4)
= u,
\end{equation*}
as required.
\end{proof}

\begin{proof}[Proof of Theorem~\ref{th-1-7}]
According to Theorem~\ref{th-2-8}, it suffices to prove that the
pullback $\HB$-line bundle $\eta \coloneqq f^* \gamma$ on $X$ admits a
stratified-algebraic structure. The Chern class
\begin{equation*}
c_2(\eta_{\CB}) = c_2(f^*\gamma_{\CB}) = f^*(c_2(\gamma_{\CB})) =
f^*(s_4)
\end{equation*}
is an adapted cohomology class in $H^4(X; \Z)$, and hence, by
Lemma~\ref{lem-2-9}, there exists a stratified-algebraic $\HB$-line
bundle $\xi$ on $X$ with
\begin{equation*}
c_2(\xi_{\CB}) = c_2(\eta_{\CB}).
\end{equation*}
Since $c_j(\xi_{\CB}) = c_j(\eta_{\CB}) = 0$ for $j=1$ and $j \geq 3$,
we get
\begin{equation*}
c_k(\xi_{\CB}) = c_k(\eta_{\CB}) \quad\textrm{for all } k \geq 0.
\end{equation*}
Hence, according to \cite[Theorem~3.2]{bib23}, the $\CB$-vector bundles
$\xi_{\CB}$ and $\eta_{\CB}$ are stably equivalent (here the assumption
on the torsion of the cohomology groups $H^{2k}(X; \Z)$ is needed).
Consequently, in view of Theorem~\ref{th-2-7}, the $\CB$-vector bundle
$\eta_{\CB}$ admits a stratified-algebraic structure. Finally, by
Theorem~\ref{th-2-6}, the $\HB$-line bundle $\eta$ admits a
stratified-algebraic structure, as required.
\end{proof}

\begin{proof}[Proof of Theorem~\ref{th-1-8}]
According to Theorem~\ref{th-2-8}, it suffices to prove that the
pullback $\HB$-line bundle $\eta \coloneqq f^*\gamma$ on $X$ admits a
stratified-algebraic structure. We modify the proof of
Theorem~\ref{th-1-7} as follows.

The Chern class
\begin{equation*}
c_2(\eta_{\CB}) = f^*(s_4)
\end{equation*}
belongs to $A^4(X; \Z)$. Note that if a cohomology class $u$ in $H^4(X;
\Z)$ is adapted, then so is $-u$. Consequently, each element in $A^4(X;
\Z)$ can be written as a finite sum of (not necessarily distinct)
adapted cohomology classes. In particular,
\begin{equation*}
c_2(\eta_{\CB}) = u_1 + \cdots + u_k,
\end{equation*}
where the $u_i$ are adapted cohomology classes in $H^4(X; \Z)$. By
Lemma~\ref{lem-2-9}, there exists a stratified-algebraic $\HB$-line
bundle $\xi_i$ on $X$ with
\begin{equation*}
c_2((\xi_i)_{\CB}) = u_i.
\end{equation*}
Since $\dim X \leq 7$, the stratified-algebraic $\HB$-vector bundle
$\theta \coloneqq \xi_1 \oplus \cdots \oplus \xi_k$ can be written as
\begin{equation*}
\theta = \xi \oplus \varepsilon,
\end{equation*}
where $\xi$ is a topological $\HB$-line bundle and $\varepsilon$ is a
trivial $\HB$-vector bundle, cf. \cite[p.~99]{bib10}. According to
Theorem~\ref{th-2-7}, the $\HB$-line bundle $\xi$ admits a
stratified-algebraic structure. Furthermore,
\begin{equation*}
c_2(\xi_{\CB}) = c_2(\theta_{\CB}) = c_2((\xi_1)_{\CB}) + \cdots +
c_2((\xi_k)_{\CB}) = c_2(\eta_{\CB}).
\end{equation*}
Since $c_j(\xi_{\CB}) = c_j(\eta_{\CB}) = 0$ for $j=1$ and $j \geq 3$,
we get
\begin{equation*}
c_k(\xi_{\CB}) = c_k(\eta_{\CB}) \quad\textrm{for all } k \geq 0.
\end{equation*}
The rest of the proof is the same as that of Theorem~\ref{th-1-7}.
\end{proof}

We record the following general fact.

\begin{lemma}\label{lem-2-10}
Let $X$ be a compact nonsingular real algebraic variety and let $v$ be a
spherical cohomology class in $H^p(X; \Z)$, where $p \geq 1$. Then the
cohomology class $2v$ is adapted. In particular,
\begin{equation*}
2 \Hsph^p (X;\Z) \subseteq A^p (X;\Z).
\end{equation*}
\end{lemma}

\begin{proof}
It suffices to prove the first assertion. Recall that $v = \llbracket M
\rrbracket^X$, where $M$ is a compact smooth codimension $p$ submanifold
of $X$, with trivial and oriented normal bundle. There exists an
isotopic copy $M'$ of $M$ such that $M \cap M' = \varnothing$ and the
union $M \cup M'$ is the boundary of a compact smooth submanifold with
boundary, embedded in $X$ with trivial normal bundle. It follows that
the smooth submanifold $M \cup M'$ is isotopic to a nonsingular Zariski
closed subvariety~$V$ of $X$ (cf. for example \cite[Lemma~2.3]{bib13}).
By construction,
\begin{equation*}
2v = 2 \llbracket M \rrbracket^X = \llbracket V \rrbracket^X,
\end{equation*}
provided that the normal bundle to $V$ in $X$ is suitably oriented.
Hence the cohomology class $2v$ is adapted, as required.
\end{proof}

\begin{proof}[Proof of Theorem~\ref{th-1-9}]
Since
\begin{equation*}
\rho(A^4(X;\Z)) = A^4(X;\Z/2) \quad\textrm{and}\quad \rho(f^*(s_4)) =
f^*(\bar{s}_4) \in A^4(X;\Z/2),
\end{equation*}
by the universal coefficient theorem, the cohomology class $f^*(s_4)$
can be written as
\begin{equation*}
f^*(s_4) = u + 2v,
\end{equation*}
where $u \in A^4(X;\Z)$ and $v \in H^4(X;\Z)$. Making use of the
equality $\Hsph^4(X; \Z) = H^4(X; \Z)$ and Lemma~\ref{lem-2-10}, we get
$f^*(s_4) \in A^4(X; \Z)$. The proof is complete in view of
Theorem~\ref{th-1-8}.
\end{proof}

\cleardoublepage
\phantomsection

\end{document}